\documentclass[12pt,psamsfonts]{amsart}
\usepackage{amsmath}
\usepackage{amsthm}
\usepackage{amssymb}
\usepackage{amscd}
\usepackage{amsfonts}
\usepackage{amsbsy}
\usepackage{color}
\usepackage{graphicx}
\usepackage{graphics}
\usepackage[normalem]{ulem}
\newcommand{\RR}{\mathbb R}

\newcommand{\R}{\mathbb{R}}
\newcommand{\be}{\begin{equation}}
\newcommand{\ee}{\end{equation}}
\newcommand{\Id}{\mathrm{Id}}
\newcommand{\tr}{\mathrm{tr}}

\def\bea#1\eea{\begin{align}#1\end{align}}
\def\non{\nonumber}

\newcommand{\Qb}{Q_\xi}
\newcommand{\Rb}{R_\xi}
\newcommand{\tC}{\tilde C}
\newcommand{\bC}{\bar C}
\newcommand{\mcR}{\mathcal{R}}
\newcommand{\sS}{\mathcal{S}}

\newcommand{\mcD}{\mathcal{D}}
\newcommand{\diag}{\textrm{diag}}
\newtheorem{theorem}{Theorem}[section]
\newtheorem{lemma}[theorem]{Lemma}
\newtheorem{remark}{Remark}
\newtheorem{proposition}[theorem]{Proposition}
\newtheorem{corollary}[theorem]{Corollary}

\begin{document}
\title[Shear flow dynamics]
{Shear flow dynamics\\ in the Beris-Edwards model of nematic liquid crystals}

\author[Adrian C. Murza , Antonio E. Teruel and Arghir D. Z\u arnescu]
{Adrian C. Murza, Antonio E. Teruel and Arghir D. Z\u arnescu}
\address{Adrian C. Murza, \\ Institute of Mathematics ``Simion Stoilow''
of the Romanian Academy, Calea Grivi\c tei 21, 010702 Bucharest, Romania}
\email{adrian\_murza@hotmail.com}
\address{Antonio E. Teruel, \\ Departament de Matem\`{a}tiques i Inform\`atica, Universitat de les Illes
Balears, Crta. de Valldemossa km. 7.5, 07122 Palma de Mallorca, Spain}
\email{antonioe.teruel@uib.es}
\address{Arghir D. Z\u arnescu
\\ IKERBASQUE, Basque Foundation for Science, Maria Diaz de Haro 3,
48013, Bilbao, Bizkaia, Spain}\address{BCAM,  Basque  Center  for  Applied  Mathematics,  Mazarredo  14,  E48009  Bilbao,  Bizkaia,  Spain}\address{``Simion Stoilow" Institute of the Romanian Academy, 21 Calea Grivi\c{t}ei, 010702 Bucharest, Romania}
\email{azarnescu@bcamath.org}


\begin{abstract} We consider the Beris-Edwards model describing nematic liquid crystal dynamics and restrict to a shear flow and  spatially homogeneous situation. We analyze the dynamics focusing on the effect of the flow. We show that in the co-rotational case one has gradient dynamics, up to a periodic eigenframe rotation, while in the non-co-rotational case we identify the short and long time regime of the dynamics. We express these in terms of the physical variables and compare with the predictions of other models of liquid crystal dynamics. 
\end{abstract}

\maketitle

\section{Introduction}

Liquid crystals are a mysterious material that is still poorly understood at a basic, fundamental level, despite its impressive technological applications, particularly in liquid crystal displays. It is a material that flows like a liquid, yet it has some properties specific to solids, such as  optical properties, that are revealed for instance when passing polarised light through it. 

There exist several models that compete in  attempting to provide a description at the continuum level.  The most comprehensive models  regard the material as a complex non-Newtonian fluid, hence they use a Navier-Stokes equation describing the average velocity of the molecules, coupled with a reaction-diffusion-convection equation describing roughly the evolution of the directions of the anisotropic molecules.  As such its study is mostly related to fluid mechanics. However, because of the presence of the Navier-Stokes equations the qualitative behaviours of the model are in general very difficult to understand.

On the other hand materials scientists are interested in features of the material relevant in regimes that do not involve significantly its flow behaviour, as it is the case for instance in liquid crystal displays. Thus they study simplified models obtained most often by formally dropping the flow out of the previously mentioned complex fluid equations. The simplified models are easier to understand particularly from the point of view of obtaining qualitative predictions.

 However, it is not clear in general what is lost through this simplification and to what extent the presence of the flow significantly affects the dynamics. One setting in which one can understand the presence of the flow, widely used in the engineering studies, and in the rheological literature in determining various properties of liquid crystals, is to consider the effect of a shear flow. This is a rather well behaved flow, for which the Navier-Stokes system simplifies dramatically, yet it produces non trivial effects. Intuitively this simplification allows to describe the local behaviour of the system near a non-singular point of the velocity.

There exist several types of liquid crystals, but we will consider just the simplest and most used in practice, the nematic liquid crystals. For these we use a model well studied in the recent years \cite{ADL14, BE94, CRWX15, PZ12}, that combines analytical tractability with physical relevance, namely the Beris-Edwards model. It is a system for the unknowns  $\mathbf{u}(x, t): \RR^3\times (0, +\infty) \rightarrow \RR^3$
representing the incompressible fluid velocity field (of the liquid crystal molecules) and $Q(x, t):
\RR^3\times (0, +\infty) \rightarrow \sS_0^{(3)}$ standing for the
order parameter of liquid crystal molecules, where we denote by $\sS_0^{(3)}$ the $Q$-tensor space
$$
  \sS_0^{(3)}:=\big\{Q\in\mathbb{R}^{3\times 3}| \, Q^{ij}=Q^{ji}, \,\forall 1\leq i, j\leq 3, \, \tr(Q)=0
  \big\}.
$$

Then the Beris-Edwards system, in non-dimensional form, is:

\begin{align*}
\mathbf{u}_t+(\mathbf{u}\cdot\nabla)\mathbf{u}-\nu\Delta\mathbf{u}+\nabla{P}=&
L\nabla\cdot(Q\Delta{Q}-\Delta{Q}Q)-L\nabla\cdot(\nabla{Q}\odot\nabla{Q})\non\\
&-\xi L\nabla\cdot\bigg(\Delta Q(Q+\frac{1}{3}Id)+(Q+\frac{1}{3}Id)\Delta Q-2(Q+\frac{1}{3}Id)(Q:\Delta Q) \bigg)\non\\
&-2\xi\nabla\cdot\bigg(\big(Q+\frac{1}{3}\Id\big)(\frac{\partial F}{\partial Q}:Q-\frac{\partial F}{\partial Q})\bigg)\\
\nabla\cdot\mathbf{u}&=0,\label{incomp}
\end{align*} where we denoted $A:B:=\textrm{tr}(AB)$ and $Id$ is the $3\times 3$ identity matrix. This is an equation for  the flow   $\mathbf{u}$ representing the local average velocity of the centers of mass of the rod-like molecules. It is a Navier-Stokes equations with an additional stress tensor encoding the non-Newtonian effect that the interaction of the particles has on their motion.

On the other hand the local orientation of the molecules, represented by $Q$, is transported by the flow, rotated and aligned by the flow, and also driven by the bulk free energy $F$ of the molecules as well as being subjected to Brownian motion:
\begin{align*}
Q_t+(\mathbf{u}\cdot\nabla )Q=&(\xi D+W)\big(Q+\frac13 Id
\big)+\big(Q+\frac13 Id \big)(\xi
D-W)\non\\
&-2\xi\big(Q+\frac13 Id \big)\tr(Q\nabla\mathbf{u})+\Gamma\left(L\Delta{Q}-\frac{\partial F}{\partial Q}\right).
\end{align*} where $F:=\frac{a}{2}\tr(Q^2)-\frac{ b}{3}\mathrm{tr}(Q^3)+\frac{ c}{4}\tr^2(Q^2)$ and
\begin{equation}\label{def:partialFQ}
\frac{\partial F(Q)}{\partial Q}=aQ-b(Q^2-\frac{1}{3}Id |Q|^2)+cQ|Q|^2
\end{equation} is its gradient in  $\sS_0^{(3)}$  (with the term $-\frac{1}{3}Id |Q|^2$ representing a Lagrange multiplier accounting for the trace-free constraint).  The parameters $b,c$ represent material-dependent constants and $a$ depends on material and temperature. We will assume throughout the following restrictions (see \cite{newtonmottram} ):

$$
b,c>0, b^2-24ac>0
$$

The matrices $D:=\dfrac{\nabla\mathbf{u}+\nabla^T\mathbf{u}}{2}$,
$W:=\dfrac{\nabla\mathbf{u}-\nabla^T\mathbf{u}}{2}$ denote the
symmetric and skew-symmetric parts of the  velocity matrix,
respectively.  The coefficient $a\in \R$ is material and temperature dependent, while $b\in \R$ and $c\in\R_+$ are only material dependent, see \cite{newtonmottram}. The parameter $\xi$ is related to the aspect ratio of the liquid crystal molecules and heuristically speaking it quantifies  the ratio between two effects that the flow has on the liquid crystal molecules: the rotating effect, related to the term $[W,Q]$ and the aligning effect that is related to the terms with $\xi$ in front.

We restrict  now  by taking  $\mathbf{u}(x,y,z)=(2y,0,0)$ to be a shear flow and $Q$ homogeneous in space. This setting is often considered in the rheological literature and captures the statistical aspects of the flow, in particular combining the dynamical aspects with the phase transitions effects that the nonlinearity in $Q$ is capable to describe.  Then we have:
\begin{equation}\label{mat}
W=\left(
\begin{array}{ccc}
0&1&0\\-1&0&0\\0&0&0
\end{array}
\right)
\hspace{2cm}
D=\left(
\begin{array}{ccc}
0&1&0\\1&0&0\\0&0&0
\end{array}
\right),
\end{equation} the equation for  $\mathbf{u}$  is trivially satisfied (with $P\equiv 0$) while the equation for $Q$ reduces to (where we will take for simplicity $\Gamma=1$):

\bea\label{eqgen1}
\frac{d}{dt}Q &=[W,Q]+\xi[DQ+QD]+\frac{2\xi}{3}D-2\xi\left(Q+\frac{1}{3}Id\right)\mathrm{tr}[QD]-\frac{\partial F(Q)}{\partial Q}\non\\
&=[W,Q]+\xi[DQ+QD]+\frac{2\xi}{3}D-2\xi\left(Q+\frac{1}{3}Id\right)\mathrm{tr}[QD]- \left(aQ-b(Q^2-\frac{1}{3}Id |Q|^2)+cQ|Q|^2\right)
\eea with $[W,Q]$ denoting the commutator of $W$ and $Q$, i.e. $[W,Q]=WQ-QW$. We assume without loss of generality that our model has been non-dimensionalised (which can be done in a standard way, completely analogously as in \cite{sluckin}) so all of our parameters are non-dimensional.

 It is worth comparing the above equation with the shear-flow model considered for instance in \cite{sluckin,chill}:

\begin{equation}\label{eqgenfake}
\begin{array}{l}
\displaystyle{\frac{d}{dt}\bar Q=\delta([W,Q]+\gamma D)-\frac{\partial F(Q)}{\partial Q}}\end{array}
\end{equation}

It was shown in \cite{sluckin,chill} that the model predicts a 
certain anomalous nongeneric continua of equilibria, the existence of these
continua shows that the model is structurally unstable. Our model is more nonlinear and will present a physically more realistic behaviour, in particular it will be seen that asymptotically the effect of the flow disappears and one obtains evolutions towards the steady state of the case without flow. 

In our case the major difference it between the case when $\xi=0$ and $\xi\not=0$.  The case $\xi=0$ is called in the literature the ``co-rotational case" and is known to be much simpler. In fact we will see that it amounts to a combination of rotation in time and gradient flow behaviour. More precisely the flow will rotate the eigenframe of the matrices with an explicit rate of rotation while the non-trivial dynamics will occur just at the level of eigenvalues but not at the level of eigenframe.  

 The case $\xi\not=0$ is much more complex and is responsible with an extraordinary  increase in the complexity of the dynamics. 
As such we will focus on understanding too asymptotic regimes, the short-time and the long-time regimes. It turns out that these are related to the size of $\xi$ in a certain sense, to be detailed later. 

The paper is organised as follows: in Section $2$ we consider the co-rotational case ($\xi=0$) and we will show that one can completely understand the dynamics and in particular obtain periodic in time solutions. In Section $3$ we will analyze the non-corotational case and focus on the short and long time regimes. We will see that we can obtain the phase portrait of the short time regime, while the long time dynamics reduce to evolution towards the manifold of stationary states (of the case without flow). In Section $4$ we discuss the results previously obtained in terms of the usual physical variables (scalar order parameters and the director) while commenting on their relevance and the degeneracies one has using them. Finally in  Section $5$  we provide a conclusion summarising the results thus obtained.

{\bf Notations and conventions:} We denote by  $\diag (a,b,c)$ the three-by-three diagonal matrix with elements $a,b$ and $c$ (from top to bottom line).  For $n,m\in\R^3$ we let $n\otimes n$ to be the three-by-three matrix with $n_im_j$ as the $ij$-th component.   We denote by $Id$ the three-by-three identity matrix and by $Id_2=e_1\otimes e_1+e_2\otimes e_2$ where $e_1=(1,0,0),e_2=(0,1,0)$.

For two three-by-three matrices $A$ and $B$ we take the scalar product in the space of matrices to be $A:B=\tr(B^tA)$ which produces the ``Frobenius norm": $|A|=\sqrt{\textrm{tr}(A^2)}$.

\section{The co-rotational case ($\xi=0$): gradient dynamics and their ``rotated" version}
\label{sec:corotational}
In this section we restrict ourselves to studying the case $\xi=0$, which is a limit case, in which the dynamics simplify significantly. In this case the equation \eqref{eqgen1} becomes

\begin{equation}\label{eqgen2}
\begin{array}{l}
\displaystyle{\frac{d}{dt}Q=[W,Q]- \frac{\partial F(Q)}{\partial Q}}=[W,Q]-\left(aQ-b(Q^2-\frac{1}{3}Id |Q|^2)+cQ|Q|^2\right)
\end{array}
\end{equation}

Then we have that up to a time dependent  rotation of the eigenvectors, the dynamics involve just the evolution of eigenvalues. More precisely we have the following:

\begin{proposition}
Consider the equation \eqref{eqgen2} where $a,b\in \R$ and $ c\in\R_+$, with $W$ as defined in \eqref{mat}. Then letting $B(t)=e^{Wt}=\left(\begin{array}{lll} \cos t & \sin t & 0 \\ -\sin t & \cos t & 0 \\ 0 & 0 & 0\end{array}\right)$ and denoting $U(t):=B^t(t)Q(t)B(t)$ we have that  $U$ is a solution of the gradient flow system:

\begin{equation}\label{eqgrad0}
\begin{array}{l}
\displaystyle{\frac{d}{dt}U=-\frac{\partial F(U)}{\partial U}}
\end{array}
\end{equation} for which  all $\omega$-limit points belong to the set of stationary points.

 Thus for any initial data $Q_0\in \sS_0^{(3)}$ there exists sequences $(t_k)_{k\in\mathbb{N}}$ with $t_k\to\infty$ and $s\in \{0,\frac{b+\sqrt{b^2-24ac}}{4c},\frac{b-\sqrt{b^2-24ac}}{4c}\}$ (depending on $Q_0$ and the time sequence) such that 

$$
\lim_{t_k\to\infty } |Q(t_k)-s\left(n(t_k)\otimes n(t_k)-\frac{1}{3}Id\right)|=0
$$ with $n(t)=(n_1\cos t+n_2\sin t, -n_1\sin t+n_2\cos t,0)$ for some  $n_1^2+n_2^2\le 1$ .

Moreover, if $\{n,m,p\}$ are an orthonormal family of eigenvectors of $Q_0$ then $\{B(t)n,B(t)m, B(t)p\}$ are an orthonormal family of eigenvectors of $Q(t)$, for any $t\ge 0$.
\end{proposition}

\begin{proof} We introduce the rotation operators $B:\R\to O(3)$ as the solution of the system:

\be\label{eq:B}
\left\{\begin{array}{l}
\frac{d}{dt}B(t)=WB(t)\\
B(0)=Id
\end{array}\right.
\ee In order to see that $B(t)\in O(3)$ for all $t\ge 0$ we note that we have: $\frac{d}{dt}B^t(t)=-B^t(t)W$. Denoting $M(t):=B(t)B^t(t)$ we see that $M(t)$ is a solution of the ODE system $\frac{d}{dt}M=W M-MW$ with $M(0)=Id$. Noting that $Id$ is a solution of this system, by uniqueness we have $M(t)=Id$ for all $t\ge 0$. In fact, one can check that we have $B(t)=e^{Wt}=\left(\begin{array}{lll} \cos t & \sin t & 0 \\ -\sin t & \cos t & 0 \\ 0 & 0 & 0\end{array}\right)$.

We further denote $U(t):=B^t(t)Q(t)B(t).$ Then \eqref{eqgen2} becomes
\bea
\frac{d}{dt}U(t)&=(\frac{d}{dt}B^t)QB+B^t(\frac{d}{dt}Q)B+B^tQ(\frac{d}{dt}B)=\left(B^t\left(\frac{d}{dt}Q-W Q+QW\right)B\right)(t)\non\\
&=\left(B^t\left(-aQ+b[Q^2-\frac{1}{3}| Q|^2Id]-c Q| Q|^2\right)B\right)(t)\non\\
&=\left(-aU+b[U^2-\frac{1}{3}|U|^2Id]-cU| U|^2\right)(t)\non\\
&=- \frac{\partial F(U)}{\partial U}\label{eq:paths}
\eea hence the equation for $U$ is a gradient type equation.

We multiply  \eqref{eq:paths} by $2 U$ and take the trace. Denoting $\alpha(t):=|U(t)|^2$ we obtain:
 
\be\label{est:alpha}
\alpha'(t)=-2a\alpha-2b \textrm{tr}(U^3)-2c \alpha^2
\ee

We note that if the matrix  $U$ has eigenvalues $\lambda,\mu$ and $-\lambda-\mu$ then $\alpha=2(\lambda^2+\mu^2+\lambda\mu)$ and $\textrm{tr}(U^3)=-3\lambda\mu(\lambda+\mu)$ hence we have:

$$
|\textrm{tr}(U^3)|\le\frac{3}{2}(\frac{\epsilon}{4}\alpha^2+\frac {1}{\epsilon}\alpha)
$$ for any $\epsilon>0$.  Using this last relation together with \eqref{est:alpha} we obtain

\[
\alpha'(t)\le -2a\alpha+3|b|(\frac{\epsilon}{4} \alpha^2+\frac{1}{\epsilon}\alpha)-2c\alpha^2
\]

Taking $\epsilon=\frac{4c}{3|b|}$ and denoting $\delta:=\frac{9b^2}{4c}-2a$ the last inequality leads to

$$
\alpha'\le \delta \alpha-c\alpha^2\le -|\delta|\alpha+M
$$ where $M=\left(\frac{\delta-|\delta|}{2\sqrt{c}}\right)^2$.

Multiplying by $e^{|\delta|t}$ and integrating over $[0,t]$ we have:
$$
\alpha(t)\le \alpha (0)e^{-|\delta|t}+\frac{M}{|\delta|}
$$ thus the trajectories are bounded, and then the general theory of gradient systems allows to conclude that the limit points are critical points of $F$ (see for instance \cite{hale}). On the other hand it is known (see for instance \cite{maj}) that critical points of $F$ are in the set $\{s\left(n\otimes n-\frac{1}{3}Id\right),\, n\in\mathbb{S}^2,s\in\{0,\frac{b\pm\sqrt{b^2-24ac}}{4c}\}.$

We continue by claiming  that the dynamics of \eqref{eqgrad0} affects only the eigenvalues, but not the eigenvectors. The argument follows the one in \cite{IXZ} and is presented here for completeness. Indeed, let us consider the system:
\begin{align}\label{eigensystem}
\frac{d\lambda}{dt}&=-\lambda\big[2c(\lambda^2+\mu^2+\lambda\mu)+a\big]+ b\Big(\frac{1}{3}\lambda^2
-\frac{2}{3}\mu^2-\frac{2}{3}\lambda\mu\Big),\nonumber\\
\frac{d\mu}{dt}&=-\mu\big[2c(\lambda^2+\mu^2+\lambda\mu)+a\big]
+ b\Big(\frac{1}{3}\mu^2-\frac{2}{3}\lambda^2-\frac{2}{3}\lambda\mu\Big).
\end{align}
The right hand side of the system is a locally Lipschitz function so
the system has a solution locally in time (in fact with some more
work global  in time and bounded, using arguments similar to the
ones before for the matrix system).

On the other hand, let us  now take an initial data
$
  Q_0=\textrm{diag}( \lambda^0,\mu^0-\lambda^0-\mu^0)
$
and denote $
  \bar Q(t):=\diag (\lambda(t) , \mu(t), -\lambda(t)-\mu(t)).
$
Then, if $\lambda(t),\mu(t)$ are solutions of
\eqref{eigensystem} with initial data $(\lambda^0,\mu^0)$
then $\bar Q(t)$ is a solution of 	\eqref{eqgrad0} with initial
data $Q_0$. On the other hand, by uniqueness of solutions of
\eqref{eqgrad0}, it must be the only solution corresponding to
the diagonal initial data $Q_0$. Thus we have shown that a diagonal
initial data will generate a diagonal solution.

For an arbitrary, non-diagonal initial data $\tilde Q_0$, since
$\tilde Q_0$ is a symmetric matrix, there exists a matrix $R\in
O(3)$, such that
$
  R\tilde
  Q_0R^t=\diag (\tilde\lambda^0, \tilde\mu^0,
   -\tilde\lambda^0-\tilde\mu^0),
$
where $(\tilde \lambda^0,\tilde
\mu^0,-\tilde\lambda^0-\tilde\mu^0)$ are the
eigenvalues of $\tilde Q_0$. If $Q(t)$ is a solution of
\eqref{eqgrad0} with initial data $\tilde Q_0$, then
multiplying on the left by {\it the time independent matrix }$R$,
and on the right by {\it the time independent matrix} $R^t$, using
the fact that $RR^t=\Id$ (as $R\in O(3)$), we obtain the following
equation:
\begin{align}
\frac{d}{dt}RQ(t)R^t=-& aRQ(t)R^t+ b\left(RQ(t)R^tRQ(t)R^t-\frac{1}{3}\tr (RQ(t)R^tRQ(t)R^t)\Id\right)\non\\
&-cRQ(t)R^t\tr\big(RQ(t)R^tRQ(t)R^t\big).\non
\end{align}
Hence if we denote by $M(t):=RQ(t)R^t$, we conclude that $M$
satisfies equation \eqref{eqgrad0}
 with initial data
$ M_0:= R\tilde Q_0R^t=\diag (\tilde\lambda^0 , \tilde\mu^0,
    -\tilde\lambda^0-\tilde\mu^0)$. Since the initial data is diagonal, we infer by previous arguments
that $M(t)$ is diagonal for all times and $
   M(t)=\diag (\lambda(t), \mu(t), -\lambda(t)-\mu(t)),
$
with $\lambda(t),\mu(t)$ solutions of \eqref{eigensystem}
with initial data $(\tilde\lambda^0,\tilde\mu^0)$. Thus we
obtain that
$
  M(t)=RQ(t)R^t=\diag (\lambda(t), \mu(t), -\lambda(t)-\mu(t)),
$
hence
$
  Q(t)=R^t\diag (\lambda(t) , \mu(t) ,
  -\lambda(t)-\mu(t))R
$ which proves our claim concerning the eingevectors of $Q(t)$.

\end{proof}

Out of the previous proof one can obtain in particular the existence of solutions that are periodic in time:

\begin{corollary}
Consider the system \eqref{eqgen2} with $b,c>0$ and $b^2-24ac>0$.  Let $Q_0:=s_\pm\left(n\otimes n-\frac{1}{3}Id\right)$ with $n\in\mathbb{S}^2$ and $s_\pm=\frac{b\pm\sqrt{b^2-24ac}}{4c}$.  Then the solution of \eqref{eqgen2} with initial data $Q_0$ is periodic in time and given by:

$$
Q(t)=s_\pm\left(n(t_k)\otimes n(t_k)-\frac{1}{3}Id\right)
$$ with  $n(t)=(n_1\cos t+n_2\sin t, -n_1\sin t+n_2\cos t,0)$.
\end{corollary}

\section{The non-corotational case ($\xi\not=0$): the short and long time regimes}

\subsection{Identifying the time regimes: a numerical insight and the magnitude the $\xi$ }

We can reprezent $Q$ in coordinates as

\be
Q:=\left(\begin{array}{lll} x & z & v \\ z & y & w \\ v & w & -x-y \end{array}\right)
\ee

We simplify and  take $v=w=0$ (which is a physically relevant regime, see Section~\ref{sec:physics}, consistent with the equations ), and then the system  \eqref{eqgen1} reduces to:


\bea\label{reduced:odesystem}
\dot x&=\frac{2}{3}(1-6x)z\xi+2z-ax+\frac{b}{3} (x^2-2xy-2y^2+z^2)-2cx(x^2+y^2+z^2+xy)\non\\
\dot y&=\frac{2}{3}(1-6y)z\xi-2z-ay+\frac{b}{3}(-2x^2-2xy+y^2+z^2)-2cy(x^2+y^2+z^2+xy)\non\\
\dot z&=(\frac{2}{3}+x+y-4z^2)\xi-x+y-az+b(xz+yz)-2cz(x^2+y^2+z^2+xy)
\eea

We consider now a couple of simulations, to understand the dynamics provided by this system. These are obtained by taking 
$a=-0.2,b=0.1,c=0.1$ and time $t=50$ but varying $\xi$. 

\begin{figure}[h]\label{fig:dyn}
\centering
\begin{tabular}{c}
\includegraphics[scale=0.5]{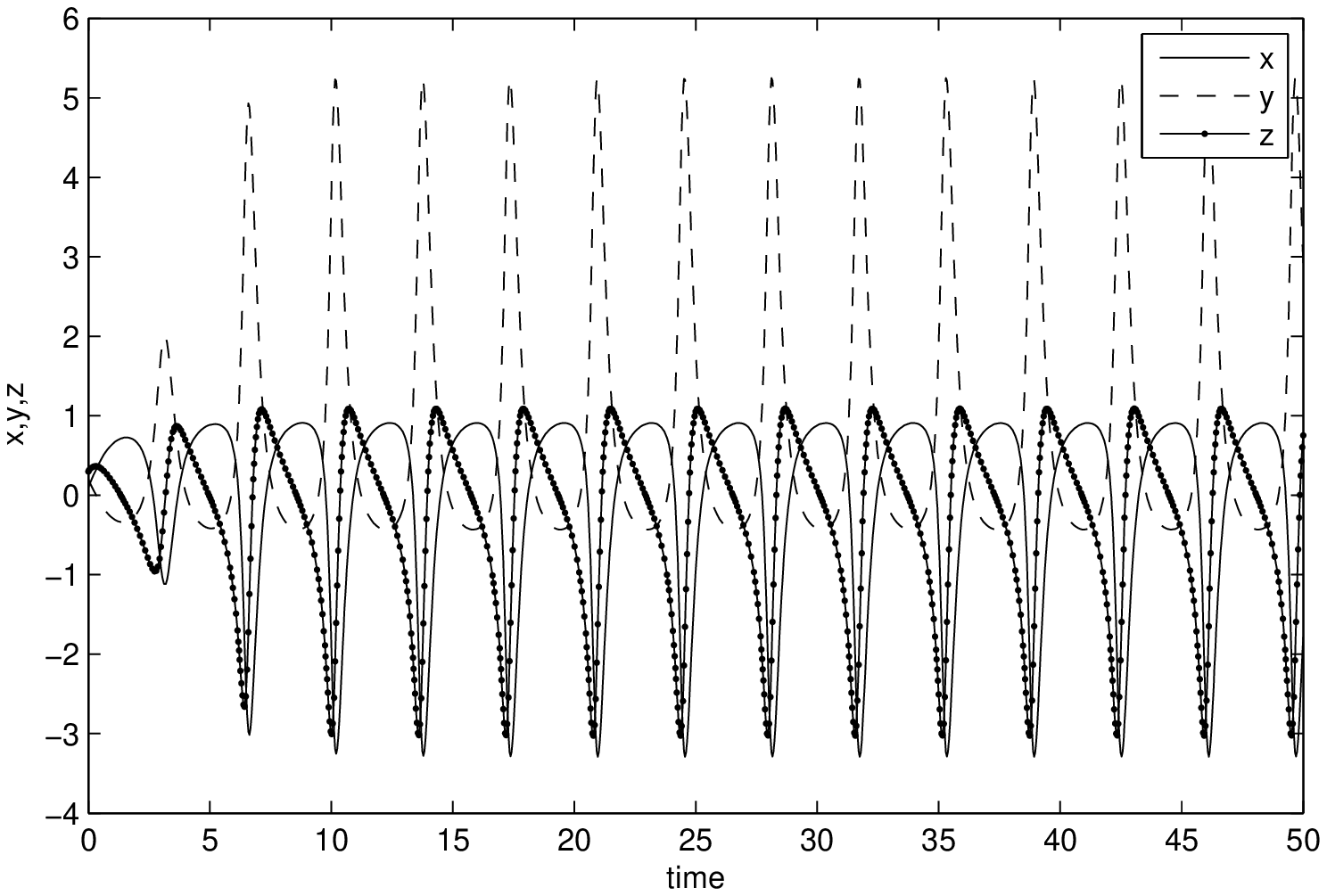}
\\\includegraphics[scale=0.5]{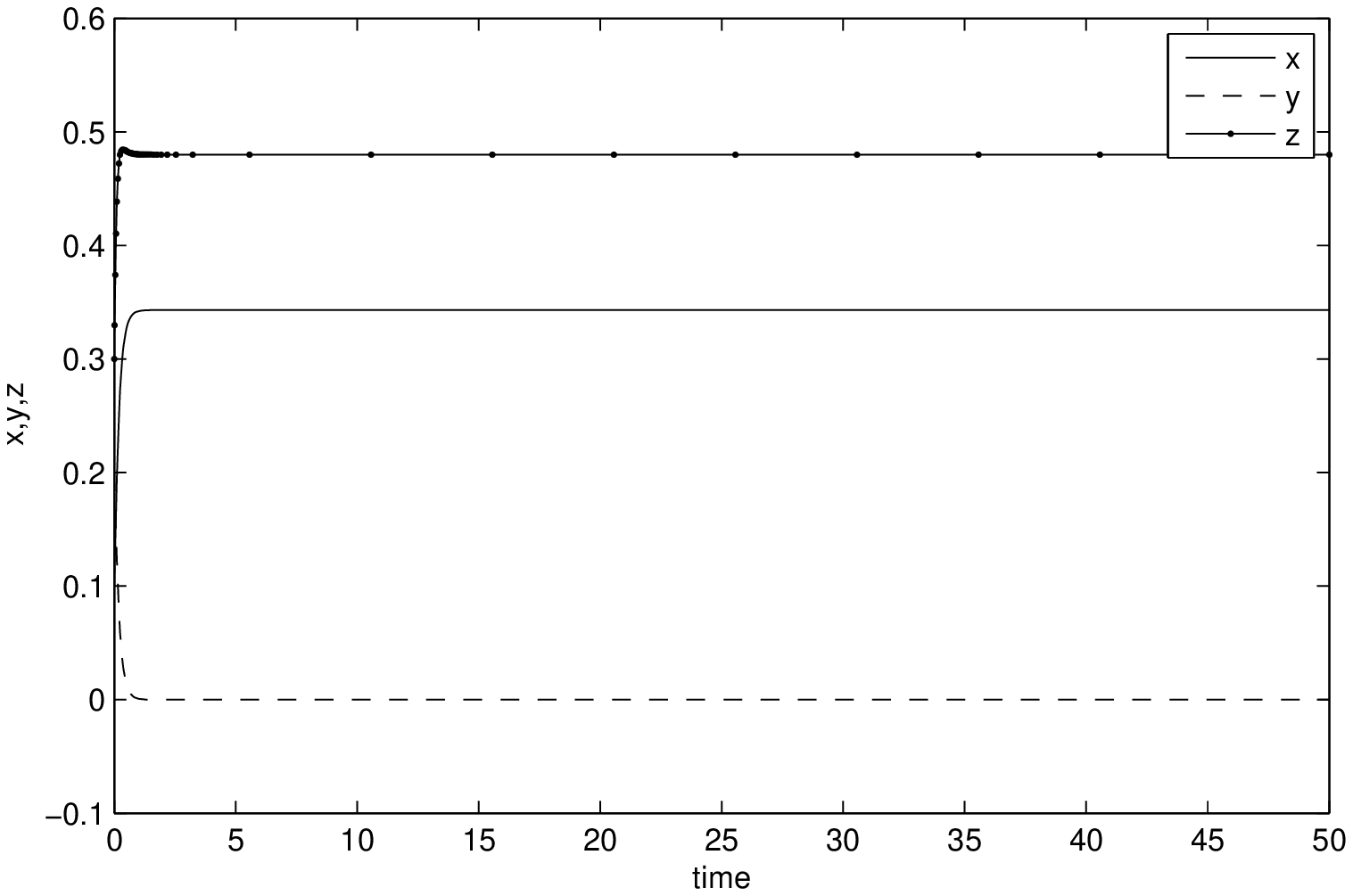}
\end{tabular}
\caption{Dynamics for $\xi=0.5$ (top) and $\xi=3$ (bottom)}
\end{figure}

What we obtain is that the dynamics are dramatically different depending on the size of $\xi$. For small $\xi$ we have essentially dynamics as provided in the co-rotational case, i.e. evolving to periodic solutions, while for large $\xi$ the periodicity is destroyed and we evolve fast towards a steady state. 

In order to obtain an analytical inside into this, let us consider the new function:

\be
Q_{\xi}(t):=Q(\frac{t}{\xi})
\ee

Then $\Qb$ satisfies the system:

\bea
\frac{d}{dt}\Qb=&[D\Qb+\Qb D]+\frac{2}{3}D-2\left(\Qb+\frac{1}{3}Id\right)\mathrm{tr}[\Qb D]+\frac{1}{\xi}\bigg([W,\Qb]-\frac{\partial F}{\partial Q}(\Qb)\bigg)\nonumber\\
=&[D\Qb+\Qb D]+\frac{2}{3}D-2\left(\Qb+\frac{1}{3}Id\right)\mathrm{tr}[\Qb D]\non\\
&+\frac{1}{\xi}\bigg([W,\Qb]-a\Qb+b(\Qb^2-\frac{1}{3}  |\Qb|^2 Id)-c\Qb|\Qb|^2\bigg)\label{eq:Qbinfty} 
\eea

Thus we have two regimes:

\begin{itemize}
\item {\bf The ``short time" regime, when  $\xi\to\infty $}.  In this  case the equation $\Qb$ formally converges on finite intervals to:

\be\label{eq:initialtime}
\frac{d}{dt}R=[DR+RD]+\frac{2}{3}D-2\left(R+\frac{1}{3}Id\right)\mathrm{tr}[RD]
\ee

Note that as $\xi\to\infty$ we have $Q(\frac{t}{\xi})\to Q(0)$ so the equation \eqref{eq:initialtime} describes the behaviour near the initial time.

\item {\bf The ``long time" regime, when   $\xi\to 0 $}.  In this case the $\Qb$ formally converges to:

\be\label{eq:stationaryflow}
0=[W,Q]-\frac{\partial F(Q)}{\partial Q}
\ee so one would expect  evolution towards a fixed point, solution of this stationary equation.

\end{itemize}

\begin{remark}
It should be noted that one does not obtain through these rescalings a direct understanding of the simulations in Figure~\ref{fig:dyn} which represent evolutions at intermediary times (neither too long, nor too short), unlike in the analytical arguments that will be provided, analyzing asymptotic regimes. However the simulations are useful for showing the complexity of the intermediary time regimes. 
\end{remark}

\subsection{The short time regime}

We have the following proposition formalizing the intuition mentioned before:

\begin{proposition}
Let $\Qb$ be the solution  of \eqref{eq:Qbinfty} with $\Qb(0)=Q_0$ and $R$ the solution of \eqref{eq:initialtime} with $R(0)=Q_0$. There exists a time $T^*$ depending only on $|Q_0|$ but independent of $\xi$ such that the solutions for both \eqref{eq:Qbinfty} and \eqref{eq:initialtime} exist on $[0,T^*]$ and moreover we have:

\be
\lim_{\xi\to\infty} \sup_{t\in [0,T^*]} |\Qb(t)-R |=0
\ee
\end{proposition}

\begin{proof}

We start by deriving  a uniform local in time estimate for $R$.  We multiply \eqref{eq:initialtime} scalarly by $2R$ to get:

$$
\frac{d}{dt}|R|^2=4\tr (R^2D)+\frac 43 \tr (DR)-4\tr (RD)|R|^2
$$

Then we have:

$$
\frac{d}{dt}|R|^2\le C_1|R|^4+C_2|R|^2+C_3
$$ where $C_1,C_2,C_3>0$ are explicitly computable coefficients.

Dividing the last estimate  by $1+|R|^4$ we get:

$$
\frac{\frac{d}{dt}|R|^2}{1+|R|^4}\le C_1+\frac{C_2}{2}+C_3
$$

Integrating on $[0,t]$ we obtain:

$$ |R(t)|^2\le \frac{|R(0)|^2+\tan (C_4 t)}{1-|R(0)|^2\tan(C_4t)}
$$ where $C_4:=C_1+\frac{C_2}{2}+C_3$. The estimate is valid for $t<T^*_1$ depending on $R(0)$ and $C_4$.

Similarly we estimate $\Qb$. We multiply \eqref{eq:Qbinfty} scalarly by $2\Qb$ and obtain:

\be
\frac{d}{dt} |\Qb|^2=4\tr (\Qb^2 D)+\frac{4}{3}\tr(D\Qb)-4\tr(D\Qb)|\Qb|^2-\frac{2}{\xi}\left(a|\Qb|^2-b\tr (\Qb)^3+c|\Qb|^4\right)
\ee which implies

\bea
\frac{d}{dt} |\Qb|^2&\le 4|\Qb^2||D|+\frac{4}{3}|D||\Qb|-4\tr(D\Qb)|\Qb|^2-\frac{2}{\xi}\left(a|\Qb|^2-b\tr (\Qb)^3+c|\Qb|^4\right)\non\\
&\le 2|\Qb^2|^2+2|D|^2+\frac 23 |\Qb|^2+\frac 23 |D|^2+2(\tr(D\Qb))^2+2|\Qb|^4-\frac{2}{\xi}\left(a|\Qb|^2-b\tr (\Qb)^3+c|\Qb|^4\right)\non\\
&\le 4|\Qb|^4+\frac{2}{3}|\Qb|^2+\frac{11}{3}|D|^2+|D|^4-\frac{2}{\xi}\left(a|\Qb|^2-b\tr (\Qb)^3+c|\Qb|^4\right)\non\\
&\le (\tC_1+\frac{\tC_2}{\xi})|\Qb|^4+\frac{\tC_3}{\xi}|\Qb|^2+\tC_4
\eea where we used that $|\tr(AB)|\le |A||B|$ for any $3\times 3$ matrices, and $|M|^4=2|M^2|^2$ for a {\it Q-tensor}, with $\tC_1,\tC_2,\tC_3,\tC_4>0$ explicitly computable constants, where  $\tC_2,\tC_3$ depend on the coefficients $a,b$ and $c$.

Dividing the last estimate  by $1+|\Qb|^4$ we get:

\be
\frac{\frac{d}{dt}|\Qb|^2}{1+|\Qb|^4}\le \tC_1+\frac{1}{\xi}(\tC_2+\frac{\tC_3}{2})+\tC_4
\ee which integrating on $[0,t]$ gives:

\be
|\Qb|^2(t)\le \frac{|Q_0|^2+\tan (\tC_{5,\xi}t)}{1-|Q_0|^2\tan(C_{5\xi}t)}
\ee where $\tC_{5,\xi}=\tC_1+\frac{1}{\xi}(\tC_2+\frac{\tC_3}{2})+\tC_4$ and for $t<T^*_2$ with $T^*_2$ depending on $|\Qb (0)|$, $\tC_1,\tC_4$  but independent of $\xi$.

We now denote $\Rb:=\Qb-R$ where $R$ is a solution of the equation \eqref{eq:initialtime} with initial data $R(0)=\Qb(0)=Q_0$. We have that $\Rb$  satisfies the equation:

\bea
\frac{d}{dt}\Rb=[D\Rb+\Rb D]-& 2\Rb\tr (\Qb D)-\frac{2}{3}Id\tr (\Rb D)-2R\tr (\Rb D)\non\\
&+\frac{1}{\xi}\bigg([W,\Qb]-a\Qb+b(\Qb^2-\frac{1}{3} Id |\Qb|^2)-c\Qb|\Qb|^2\bigg)
\eea

Multiplying by $2\Rb$ and estimating similarly as before we obtain:

\be\label{eq:Rb}
\frac{d}{dt} |\Rb|^2\le \bC_1 (1+\frac{1}{\xi})|\Rb|^2+\bC_2 |\Rb|^2 |\Qb| +\bC_3 |\Rb|^2 |R|+\frac{\bC_4}{\xi}(1+|\Qb|^6)
\ee where $\bC_1,\bC_2,\bC_3,\bC_4>0$ are explicitly computable constants depending only on  $D$, $W$, $a,b,c$ but not on $\xi$.

Noting that  $\Qb(0)=Q_0=R(0)$ we have  $\Rb(0)=0$ and thus using Gronwall inequality we get out of   \eqref{eq:Rb} that: 

\be
|\Rb|^2(t)\le\frac{\bC_4}{\xi}\int_0^t (1+|\Qb|^6)(s)\exp (\int_s^t  \bC_1(1+\frac{1}{\xi})+\bC_2|\Qb|(\tau)+\bC_3  |R|(\tau) \,d\tau)\,ds
\ee

Thus, for $T<\min\{T_1^*,T_2^*\}$ we have

\be
\sup_{t\in [0,T]} |\Qb(t)-R(t)|\to 0, \textrm{ as }\xi\to\infty
\ee

\end{proof}

\subsection{Coordinates and the analysis of the short time regime}

In order to describing the initial behaviour more precisely it is convenient to revert to coordinate representation.
The  system describing the initial behaviour in coordinates is given by:

\begin{equation}\label{sysabc01}
\begin{array}{l}
\dot x=\displaystyle{\frac{2}{3}}(1-6x)z\\
\\
\dot y=\displaystyle{\frac{2}{3}}(1-6y)z\\
\\
\dot z=\displaystyle{\frac{2}{3}}+x+y-4 z^2.
\end{array}
\end{equation}

\begin{proposition}\label{propfirstint}
System \eqref{sysabc01} is integrable  in $\mathbb{R}^3\setminus \{x=1/6\}$  and the two independent first integrals are
\begin{equation}\label{fi}
\begin{array}{l}
\mathcal{H}_1(x,y,z):=\displaystyle{\frac{1-6y}{6 (1 - 6x)}}\\
\\
\mathcal{H}_2(x,y,z):=\displaystyle{\frac{1+6x+6y-12z^2}{36 (1 - 6x)^2}}.
\end{array}
\end{equation}
\end{proposition}
\begin{proof}
The proof follows from straightforward calculations.
\end{proof}

\begin{proposition}\label{propequil} Singular points of system \eqref{sysabc01} are: 
every point belonging the straight line 
$$r:= \left\{x+y=-2/3,\ z=0\right\},$$ which is non-hyperbolic; and the hyperbolic ones 
\begin{equation*}
\begin{array}{l}
E_2:=\displaystyle{\left(\frac{1}{6},\frac{1}{6},-\frac{1}{2}\right)},~~
E_3:=\displaystyle{\left(\frac{1}{6},\frac{1}{6},\frac{1}{2}\right)},~~
\end{array}
\end{equation*}
where $E_2$ is a repelling node and $E_3$ is an attracting node. The planes $x+y+2z=-2/3$ and $x+y+2z=-2/3$ are invariant under the flow and intersect along $r$. Moreover, the global phase portrait of system \eqref{sysabc01} is topologically equivalent to the one represented in Figure \ref{figspace}.
\end{proposition}

\begin{proof}
The existence of the singularities and the local behaviour of the hyperbolic ones follow from straightforward calculations.

The invariance of the planes in the statement of the proposition follows by checking that
\[
 \dot{x}\dfrac {df}{dx} + \dot{y}\dfrac {df}{dy} + \dot{z}\dfrac {df}{dz}= -2(2z \pm 1) f(x,y,z)
\]
where $f(x,y,z)=x+y-2z+2/3$ and $f(x,y,z)=x+y+2z+2/3$, respectively.

Let we describe now the global phase portrait. To do that, we first describe the foliation induced by the level surfaces $I_h$ of the first integral $\mathcal{H}_1$, and then to analyze the dynamical behaviour of the restriction of system \eqref{sysabc01} to $I_h$, we use the restriction to $I_h$ of the other first integral $\mathcal{H}_2$. 

Note that the invariant surfaces $I_h:=\{(x,y,z): \mathcal{H}_1(x,y,z)=h\}_{h\in\mathbb{R}}$, are the planes given by 
\begin{equation}\label{invmani}
y=1/6-h(1-6x) 
\end{equation}
and $z\in\mathbb{R}$. Moreover, the plane given by $x=1/6$ and $z\in \mathbb{R}$ is also invariant under the flow of system \eqref{sysabc01}, even when it is not a level surface of $\mathcal{H}_1$. All these vertical planes intersect along the invariant straight line $x=1/6$, $y=1/6$ and $z\in \mathbb{R}$.

By replacing expression \eqref{invmani} into the third equation of \eqref{sysabc01} we obtain 
\begin{equation}\label{pp3}
\begin{array}{l}
\dot{x}=\displaystyle{\frac{2}{3}(1-6x)z}\\
\\
\dot{z}=\displaystyle{\frac{5}{6}+x-4z^2-h(1-6x)},
\end{array}
\end{equation}
which corresponds with the restricted differential system over the invariant manifold $I_h$. 

If $h\neq -1/6$, then it is easy to check that the system \eqref{pp3} has three hyperbolic singular points: one of them 
\begin{equation*}
{r}_h=\left(\dfrac {6h-5}{6(1+6h)},0\right),
\end{equation*}
corresponding with the intersection of the straight line  $r$ and the plane $I_h$; and the other two corresponding with $E_2$ and $E_3$, which will be referenced in the same way. Linear analysis around the singularities assures that $r_h$ is a saddle point, $E_2$ is a repelling node and $E_3$ is an attracting node. 

For the globaly description of the phase portrait of system \eqref{pp3} we resort to the restriction of the first integral $\mathcal{H}_2$ to the invariant plane $I_h$, 
\begin{equation*}
\left. \mathcal{H}_2 \right|_h= \dfrac {1+3x+6z^2-3h(1-6x)}{18(1-6x)^2}.
\end{equation*}
Hence, the level curves of $\left. \mathcal{H}_2 \right|_h$ containing the singular point $r_h$, that is $\left. \mathcal{H}_2 \right|_h(x,z)=\left. \mathcal{H}_2 \right|_h(r_h)$, are the two invariant straight lines
\begin{equation*}
 (5+6x-12z-6h(1-6x))(5+6x+12z-6h(1-6x))=0.
\end{equation*}
Notice that the first of these lines contains the attracting node $E_2$ whereas the second one containes the repelling node $E_3$. From this, we conclude that the global phase portrait is topological equivalent to the one depicted in Figure \ref{figplano}(a) when $h<-1/6$ or in Figure \ref{figplano}(c) when $h>-1/6$.

On the other hand, if $h=-1/6$, then system \eqref{figplano} exhibits only the singularities which correspond with $E_2$ and $E_3$, having both of them the same local behaviour as before. Moreover, the straight lines 
\[
 x=\dfrac 1 6,\ z=\dfrac 1 2,\ \text{and}\ z=-\dfrac 1 2,
\]
are invariant under the flow. We conclude that the global phase portrait is topological equivalent to the one depicted in Figure \ref{figplano}(b).

\begin{figure}[hb]
\centering
\begin{center}
\includegraphics{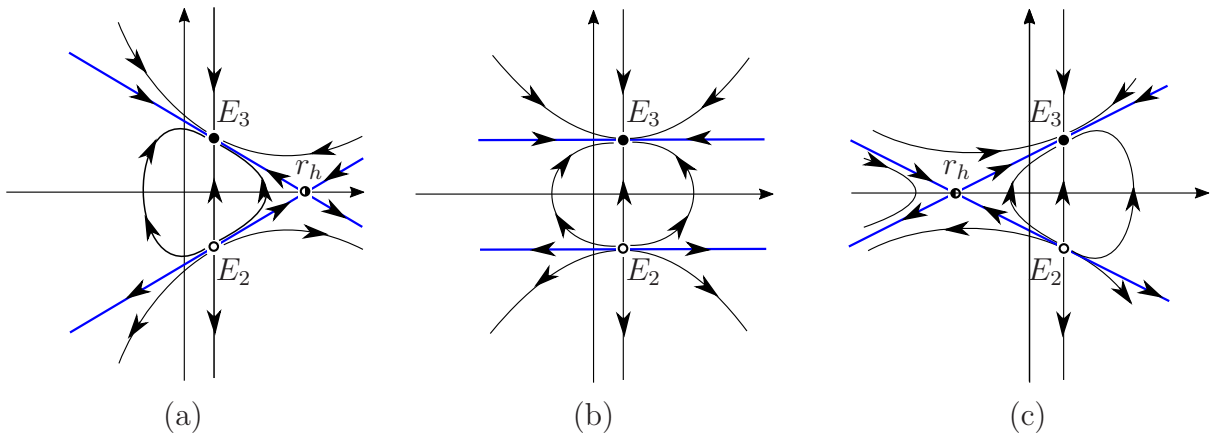}\vspace{0.35cm}
\begin{picture}(0,0)
\put(-400,-15){(a)}
\put(-380,40){$E_2$}
\put(-380,100){$E_3$}
\put(-350,80){$r_h$}
\put(-245,-15){(b)}
\put(-225,40){$E_2$}
\put(-225,100){$E_3$}
\put(-80,-15){(c)}
\put(-73,40){$E_2$}
\put(-73,100){$E_3$}
\put(-105,80){$r_h$}
\end{picture}
\caption{Schematic representation of the flow of system \eqref{pp3} in the $xz$ plane, vertical invariant line at $x=1/6$: (a) when $h<-1/6$; (b) when $h=-1/6$; and (c) when $h>-1/6$. Blue orbits correspond with the separatrices.}\label{figplano}
\end{center}
\end{figure}

To finish describing the global behaviour of system \eqref{sysabc01}, it remains to  describe the flow over the invariant plane $x=1/6$. This can be done just by notting than the restriction of the system \eqref{sysabc01} to the plane $x=1/6$ coincides with the system \eqref{pp3}, by taking $h=0$ and performing the change of variable $x \to y$. Therefore, the flow over the plane $x=1/6$ is topological equivalent to the one depicted in Figure \ref{figplano}(c).
\end{proof}

\begin{figure}[ht]
\begin{center}
\includegraphics{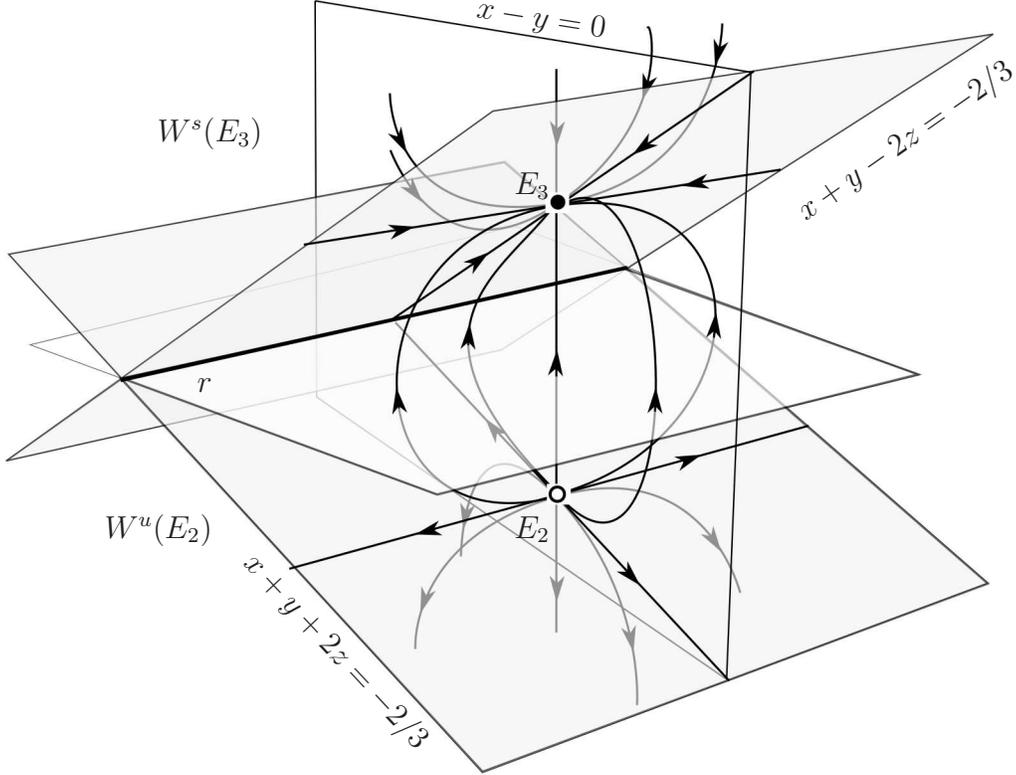}\vspace{0.35cm}
\begin{picture}(0,0)
\put(-305,145){$r$}
\put(-80,210){\rotatebox{35}{$x+y-2z=-2/3$}}
\put(-290,80){\rotatebox{-47}{$x+y+2z=-2/3$}}
\put(-200,287){\rotatebox{-9}{$x-y=0$}}
\put(-185,220){$E_3$}
\put(-320,240){$W^s(E_3)$}
\put(-340,90){$W^u(E_2)$}
\put(-185,90){$E_2$}
\end{picture}
\caption{Global phase portrait of system \eqref{sysabc01}. The thick line corresponds with the continuum of non-hyperbolic singular points $r$, the non-filled circle corresponds with the repelling node $E_2$ and the filled one with the attracting node $E_3$. The gray planes correspond with the invariant surface $x+y+2z=-2/3$, containing $E_2$ and the invariant surface $x+y+2z=-2/3$, containing $E_3$. The first of these planes is the boundary of the stable manifold of $E_3$, denoted by $W^s(E_3)$, whereas the second plane is the boundary of the unstable manifold of $E_2$, denoted by $W^u(E_2)$. We also represent the planes $z=0$ and $x-y=0$. These planes have not dynamical meaning, the second one is included just to justify that both half-spaces $x-y<0$ and $x-y>0$ are invariant under the flow, see Remark \ref{rem:invmani}.}\label{figspace}
\end{center}
\end{figure}

\begin{remark}\label{rem:invmani}
The plane $x-y=0$ is invariant under the flow of system  \eqref{sysabc01} since it corresponds with the level surface $I_{1/6}$, see \eqref{invmani}. Therefore, both half-spaces $x-y<0$ and $x-y>0$ are also invariant under the flow of system \eqref{sysabc01}.
\end{remark}

\subsection{The long time regime}
We have first an uniform estimate for $|\Qb|$ which uses in an essential manner the fact that the material constant $c$ appearing in \eqref{def:partialFQ} is {\it positive}.

\begin{lemma}
Let $\Qb$ be a solution of \eqref{eq:Qbinfty}. There exists $\tilde C, \xi_0>0$ depending on the material constants $a,b,c$  but independent of $\xi$, such that:

\be\label{est:uniformsmallbeta}
|\Qb|(t)\le |\Qb|(0)+ \tilde C
\ee for all $t\ge 0$ and $\xi<\xi_0$.

\end{lemma}

\begin{proof}
We multiply \eqref{eq:Qbinfty} by $2\Qb$ and we obtain the estimate:

\bea
\frac{d}{dt}|\Qb|^2 &\le C_1 |\Qb|^2+C_2+\frac{C_3}{\xi}|\Qb|^2-\frac{C_4}{\xi} |\Qb|^4\non\\
&\le C_2+\frac{C_5}{\xi}-\frac{C_6}{\xi}|\Qb|^2\non
\eea for some $\xi<\xi_0$ and positive explicitly computable constants $C_i>0, i=1,\dots, 6$ depending on $a,b,c$ and $\xi_0$ but not on $\xi$.

Multiplying the last inequality by $e^{\frac{C_6t}{\xi}}$ and integrating on $[0,t]$ we get:

$$
\displaystyle{|\Qb|^2(t)\le |\Qb|^2(0)e^{-\frac{C_6t}{\xi}}+\left(C_2+\frac{C_5}{\xi}\right) \frac{\xi}{C_6}\left(1-e^{-\frac{C_6t}{\xi}}\right)}
$$ out of which we obtain the claimed estimate 	\eqref{est:uniformsmallbeta}.

\end{proof}

\bigskip We show now that the long-time behaviour of $Q$ described by the limit on finite time intervals of $\Qb$ as $\xi\to 0$ is provided by solutions of \eqref{eq:stationaryflow}.

\begin{lemma}
Let $\Qb$ be a solution of \eqref{eq:Qbinfty}. We have:

\be\label{eq:limitWF}
\lim_{\xi\to 0} \int_0^t |[W,\Qb]-\frac{\partial F}{\partial Q}(\Qb)|^2(s)\,ds=0
\ee
\end{lemma}

\begin{proof}
We multiply the equation \eqref{eq:Qbinfty} by $\frac{\partial F}{\partial Q}(\Qb)-[W,\Qb]$ and obtain:

\bea
\partial_t F(\Qb)&=-\frac{1}{\xi}|\frac{\partial F}{\partial Q}(\Qb)-[W,\Qb]|^2\non\\
&+\left(\frac{\partial F}{\partial Q}(\Qb)-[W,\Qb]\right)\left([D\Qb+\Qb D]+\frac{2}{3}D-2\left(\Qb+\frac{1}{3}Id\right)\mathrm{tr}[\Qb D]\right)\non\\
&\le -\frac{1}{2\xi}|\frac{\partial F}{\partial Q}(\Qb)-[W,\Qb]|^2+\tilde C\xi
\eea for $\xi<\xi_0$ (with $\xi_0$ depending on the time $t$ and the initial data), with $\tilde C$ independent of $\xi_0$,  where for the last inequality we used the estimate \eqref{est:uniformsmallbeta}. Integrating the last inequality and using the fact that $F(\Qb)$ is bounded for bounded $\Qb$ we obtain the claimed relation \eqref{eq:limitWF}.

\end{proof}

We continue by identifying the solutions of the limiting equation \eqref{eq:stationaryflow}:

\begin{lemma}
For the matrix $W$ as in \eqref{mat} and $\frac{\partial F}{\partial Q}(Q) $ as in \eqref{def:partialFQ} we have

$$[W,Q]-\frac{\partial F}{\partial Q}(Q)=0$$ if and only if

\be
Q\in \{ 0, s_\pm (n\otimes n-\frac 13 Id), i=1,2,3, n\in\mathbb{S}^2\}
\ee with $\displaystyle{s_\pm:=\frac{b\pm \sqrt{b^2-24ac}}{4c}}$.

\end{lemma}

\begin{proof}
For a given tensor $Q$ we take $\mathcal{R}[Q]\in SO(3)$ to be such that $\mathcal{R}[Q] Q \mathcal{R}^t[Q]$ is a diagonal matrix, say

\be
\mcD:=\mathcal{R}[Q] Q \mathcal{R}^t[Q]=\left(\begin{array}{lll} x & 0 & 0 \\ 0 & y & 0 \\ 0 & 0 & -x-y\end{array} \right)
\ee with $x,y,-x-y$ being eigenvalues of $Q$.

 We explicitate

$$\mathcal{R}[Q]:=\left(\begin{array}{lll} a_{11} & a_{12} & a_{13} \\ a_{21} & a_{22} & a_{23}\\
a_{31} & a_{32} & a_{33}\end{array}\right)$$ (note that the coefficients depend on the matrix $Q$). Then, denoting

\be
\alpha:=a_{11}a_{22}-a_{12}a_{21},\, \xi:=a_{11}a_{32}-a_{12}a_{31},\, \gamma:=a_{21}a_{32}-a_{22}a_{31}
\ee we have:
\bea
\mathcal{R}[Q] W \mathcal{R}^t[Q]&=\left(\begin{array}{lll} 0 & \alpha & \xi \\
-\alpha & 0 & \gamma \\ -\xi & -\gamma & 0 \end{array}\right),\non\\
\, \mathcal{R}[Q] (WQ-QW) \mathcal{R}^t[Q]&=
\left(\begin{array}{lll} 0 & \alpha(y-x) & -\xi (2x+y) \\    \alpha(y-x) & 0  & -\gamma(2y+x) \non\\
-\xi(2x+y) & -\gamma(2y+x) & 0   \end{array}\right).
\eea

Furthermore we note that the {\it non-diagonal} terms of $\mathcal{R}[Q]([W,Q]-\partial F[Q])\mathcal{R}^t[Q]$ are zero if and only if

$$
\alpha(y-x)=\xi(2x+y)=\gamma(2y+x)
$$ which gives the following possibilities:
\bigskip
\begin{enumerate}
\item $x=y=0$,
\item $x=y\not=0$ and $\xi=\gamma=0$,
\item $x\not=y$, $2x+y=0$, $2y+x\not=0$ and $\alpha=\gamma=0$,
\item $x\not=y$, $2x+y\not=0$, $2y+x=0$ and $\alpha=\xi=0$,
\item $x\not=y$,  $2x+y\not=0$, $2y+x\not=0$  and $\alpha=\xi=\gamma=0$.
\end{enumerate}

Let us not however that the last case cannot happen. Indeed, if  $\alpha=\xi=\gamma=0$, then $\det \mathcal{R}[Q]=a_{13}\gamma-a_{23}\xi+a_{33}\alpha=0$ which cannot happen because $ \mathcal{R}[Q]\in SO(3)$.

Thus, the remaining cases all imply that $Q$ is uniaxial, i.e. it has two equal eigenvalues. We can assume without loss of generality that $x=y$ hence $\mcD=\textrm{diag}(x,x,-2x)$ (if not there exists a rotation $\tilde\mcR$ such that $\tilde\mcR \mcD\tilde\mcR^t$ is of this form).  Then the {\it diagonal} terms of $\mathcal{R}[Q]([W,Q]-\frac{\partial F}{\partial Q}(Q))\mathcal{R}^t[Q]$ are multiple of $x(a+bx+6cx^2)$ hence they are zero if and only if $\displaystyle{x\in\left\{0, -\frac{b+\sqrt{b^2-24ac}}{12c},  -\frac{b-\sqrt{b^2-24ac}}{12c}\right\}}$.
\end{proof}

\section{Physical variables: the degree order parameters $S_1,S_2$ and the angle $\theta$ of the director}
\label{sec:physics}

The main physical characteristic of nematic liquid crystal is the local preferred orientation of the rod-like molecules. The most comprehensive modelling of this characteristic is through a probability measure $\mu(x,\cdot)$ on the unit sphere $\mathbb{S}^2$ at each point $x$ in the three dimensional container containg liquid crystal material (see \cite{dgbook,newtonmottram}). Then for $A\subset\mathbb{S}^2$ and any given point $x$ the number $\mu(x,A)$ denotes the probability of finding molecules pointing in the direction $A$. Because the molecules do not distinguish their ends (i.e. they have no head or tail) we have that

\be\label{symm:mu}
\mu(x,A)=\mu(x,-A).
\ee The major insight of De Gennes was the idea to replace the measure with a moment of it, that captures the most physical information (\cite{dgbook,newtonmottram}). Because of the symmetry \eqref{symm:mu} the first order moment $\int_{\mathbb{S}^2}p\, d\mu(x,p)=0$ hence the most significant moment is the second order moment $M(x)=\int_{\mathbb{S}^2}p\otimes p \, d\mu(x,p)$. If we take $\mu(x,\cdot)$ to be the uniform distribution on the unit sphere then $M(x)$ becomes equal to $\frac{1}{3}Id$. Thus, we can define a ``$Q$-tensor" as 

$$Q(x):=\int_{\mathbb{S}^2} p\otimes p \,d\mu(x,p)-\frac{1}{3}Id$$

The $Q$-tensor thus defined is a three by three symmetric and traceless matrix. By linear algebra, we have its spectral representation as:

\be
Q=\lambda_1n\otimes n+\lambda_2 m\otimes m+\lambda_3 l\otimes l
\ee with $\lambda_1,\lambda_2,\lambda_3$ eigenvalues of $Q$ with the corresponding $n,m,l$ as an orthonormal system of eigenvectors. Since $Q$ is traceless we have 

\be 
\lambda_1+\lambda_2+\lambda_3=0
\ee

For the eigenvectors we have:

\be
n\otimes n+m\otimes m+l\otimes l=Id
\ee

We can further assume without loss of generality that $l=e_3$. Indeed, whatever $l\in\mathbb{S}^2$ is, there exists a rotation $R\in O(3)$ that takes it into $e_3$. Then $RQR^t$ has the representation:

 \be
RQR^t=\lambda_1Rn\otimes Rn+\lambda_2 Rm\otimes Rm+\lambda_3 e_3\otimes e_3
\ee

This says that in a suitable choice of coordinates we can take one of the eigenvectors to be $e_3$. This choice has the advantage that then $Q$ is simpler, namely if $e_3$ is an eigenvector,  we can represent $Q$ in coordinates as:

 \be\label{q:rep}
Q=\left(\begin{array}{lll} x & z & 0 \\ z & y & 0 \\ 0 & 0 & -x-y \end{array}\right)
\ee

We can then denote:

\be
n=(\cos\theta,\sin\theta,0), \, m=(-\sin\theta,\cos\theta, 0), l=e_3
\ee

Then following the notations used in \cite{sluckin} we aim to represent

\be\label{Q:stheta0}
Q=\frac{3}{2}S_1 (n\otimes n-\frac{1}{3}Id)+\frac{3}{2}S_2 (m\otimes m-e_3\otimes e_3)
\ee with $n,m,e_3$ an orthonormal family of eigenvectors of $Q$. This type of representation is physically relevant particularly from an optical point of view (\cite{Virgabook}) with $n=(\cos\theta,\sin\theta,0)$ representing the optical director, the average preferred orientation of the molecules and  $S_1,S_2$ representing ``scalar order parameters" indicating the average level of ordering around the optical director, respectively perpendicular to it.

In terms of these variables Section~\ref{sec:corotational} has a clear interpretation of the dynamics: in the presence of a co-rotational flow, the dynamics amount to a rotation of the optical director and the only non-trivial dynamics occurs at the level of eigenvalues. In particular there exist periodic in time dynamics which involve just a rotation of the optical director.

However, despite the physical advantages of using the $S_1,S_2,\theta$ description, this representation has certain degeneracies associated to it. We note that unlike in the $(x,y,z)$ representation of $Q$, one cannot uniquely associate to a given matrix $Q$ a unique triple 
$S_1,S_2,\theta$. Indeed, writing $Id_2:=\left(\begin{array}{lll} 1 & 0  & 0\\ 0 & 1 & 0 \\ 0 & 0 & 0  \end{array}\right)$ and noting that $n\otimes n+m\otimes m=Id_2$ we have:

\bea\label{Q:stheta0}
Q=&\frac{3}{2}S_1(n\otimes n-\frac{1}{3}Id_2)+\frac{3}{2}S_2 m\otimes m-\frac{3}{2}(\frac{S_1}{3}+S_2)e_3\otimes e_3\\
=&\frac{3}{2}S_1(n\otimes n-\frac{1}{3}n\otimes n-\frac{1}{3}m\otimes m)+\frac{3}{2}S_2 m\otimes m-\frac{3}{2}(\frac{S_1}{3}+S_2)e_3\otimes e_3\non\\
=&S_1 n\otimes n+\frac{3}{2}(S_2-\frac{S_1}{3})m\otimes m-(\frac{3}{2}S_2+\frac{S_1}{2})e_3\otimes e_3
\eea where $n,m$ are eigenvectors of the  ``2d" matrix:

$$P:=Q+(x+y)e_3\otimes e_3=\left(\begin{array}{lll} x & z & 0 \\ z & y & 0 \\ 0 & 0 & 0 \end{array}\right).$$

  On the other hand, by the linear algebra we have:

\be\label{eq:spectral}
\left(\begin{array}{lll} x & z & 0 \\ z & y & 0 \\ 0 & 0 & 0 \end{array}\right)=\lambda_1 n\otimes n+\lambda_2 m\otimes m
\ee where $\lambda_1,\lambda_2$ are eigenvalues of $P$, while $n$ and $m$ are the corresponding eigenvectors. The representation \eqref{eq:spectral} does not uniquely determine $\lambda_1,\lambda_2$ because one can interchange the eigenvalues (with a corresponding change of eigenvectors). If we choose however an order, say $\lambda_1\ge \lambda_2$ the eigenvalues are uniquely determined. In our case this amounts to the choice:

\be\label{ass:S}
S_1\ge \frac{3}{2}(S_2-\frac{S_1}{3})\leftrightarrow S_1\ge S_2
\ee

We have a further degeneracy if $\lambda_1=\lambda_2$ namely then we can choose any $n\perp m$ as eigenvectors, so for $\lambda_1=\lambda_2$ i.e. in our notations $S_1=S_2$ we have that $\theta$ can be chosen arbitrarily.

Furthermore we have a degeneracy in our choice of $\theta$ variable. Indeed, of one replaces $\theta$ by $\theta+\pi$ we have that $n\otimes n$ respectively $m\otimes m$ are the same.  We will assume that

\be\label{ass:theta}
\theta\in [-\frac{\pi}{4},\frac{\pi}{4})
\ee
Taking into account these degeneracies the representation of $Q$ becomes:

\be
Q=\frac{3}{2}S_1 \left(\begin{array}{lll} \cos^2\theta-\frac{1}{3} & \frac{\sin(2\theta)}{2} & 0 \\
 \frac{\sin(2\theta)}{2}  & \sin^2\theta-\frac{1}{3} & 0 \\ 0 & 0 & -\frac{1}{3}\end{array}\right)+\frac{3}{2}S_2 \left(\begin{array}{lll} \sin^2\theta & -\frac{\sin(2\theta)}{2} & 0 \\
 -\frac{\sin(2\theta)}{2}  & \cos^2\theta & 0 \\ 0 & 0 & -1\end{array}\right)
\ee
hence
\bea\label{coord:tranx}
x=& \frac 32 S_1\left(\cos^2\theta-\frac 13\right)+\frac 32 S_2\sin^2\theta\nonumber\\
y=&\frac 32 S_1\left(\sin^2\theta-\frac 13\right)+\frac 32 S_2\cos^2 \theta\\
z=&\frac 34 \left(S_1-S_2\right)\sin(2\theta)\nonumber
\eea so
\be\label{coord:tranxy}
x-y=\frac 32 (S_1-S_2)\cos(2\theta), 2z=\frac 32 (S_1-S_2)\sin(2\theta).
\ee

In order to eliminate further degeneracies in $\theta $ it will be convenient to focus just on part of the phase space in $x,y,z$ namely we will assume that $x-y\neq 0$. We note that this is an invariant region of system \eqref{sysabc01}, see Remark \ref{rem:invmani}.

Therefore, we get
$$\cos(2\theta)=\frac{1}{\sqrt{1+\left(\frac{2z}{x-y}\right)^2}}$$
and then, after a couple of algebraic manipulations it follows 

\bea\label{coord:trans}
\theta&=\frac 12 \arctan \left(\frac{2z}{x-y}\right),\non\\
S_1 &=\frac{x+y}{2}+ \frac{{\rm sgn}(x-y)}{2} \sqrt{(x-y)^2+4 z^2},\\
S_2 &=\frac{x+y}{2}-\frac{{\rm sgn}(x-y)}{6} \sqrt{(x-y)^2+4 z^2},\non
\eea
where ${\rm sgn}()$ stands for the signum function.

\begin{remark} Note that the map given by \eqref{coord:tranx} is a diffeomorphism, with inverse \eqref{coord:trans}, from $U_+$ to $V_+$ and from $U_-$ to $V_-$, where
\[
\begin{array}{ll}
U_+=\left\{S_1-S_2>0\right\} \times \left(-\dfrac {\pi}{4},\dfrac{\pi}{4}\right),&
V_+=\left\{x-y>0\right\}\times \mathbb{R},\\ & \\
U_-=\left\{S_1-S_2<0\right\} \times \left(-\dfrac {\pi}{4},\dfrac{\pi}{4}\right),& 
V_-=\left\{x-y<0\right\}\times \mathbb{R}.
\end{array}
\]
Moreover, \eqref{coord:tranx} maps the plane $\{S_1-S_2=0\}$ into the straight line $\{x-y=0,\ z=0\}$; the half-plane $\{S_1-S_2>0,\ \theta=\pi/4 \}$ into the half-plane $\{x-y=0,\ z>0\}$; the half-plane $\{S_1-S_2<0,\ \theta=\pi/4 \}$ into the half-plane $\{x-y=0,\ z<0\}$; the half-plane $\{S_1-S_2>0,\ \theta=-\pi/4 \}$ into the half-plane $\{x-y=0,\ z<0\}$; and the half-plane $\{S_1-S_2<0,\ \theta=-\pi/4 \}$ into the half-plane $\{x-y=0,\ z>0\}$. 
\end{remark}

We can aim now to translate into these coordinates the short time dynamics. Namely system \eqref{sysabc01} together with \eqref{coord:tranx} provide the following equations
\begin{equation}\label{ss1}
\begin{array}{l}
\dot{S}_1=\displaystyle{\frac{1}{3}(1+3S_1)(3S_2-3S_1+2)\sin(2\theta)},\\
\\
\dot{S}_2=\displaystyle{\frac{1}{9} \left(27 S_2^2-27 S_1S_2-9S_2+3 S_1-2\right)\sin(2\theta)},\\
\\
\dot{\theta}=\displaystyle{\frac{(4+3S_1+9S_2)\cos(2\theta)}{9(S_1-S_2)}},
\end{array}
\end{equation}
in the region $\{S_1-S_2 \neq 0\} \times [-\pi/4,\pi/4]$. Nevertheless, in the next result we only describe the dynamical behaviour in the region $\{S_1-S_2 > 0\} \times [-\pi/4,\pi/4]$. In the rest of the phase space the behaviour follows similarly. System \eqref{ss1} is integrable and the two independent first integrals are

\begin{equation}\label{ss2}
\begin{array}{l}
\mathcal{V}_1=\displaystyle{\frac{-2+3S_1+9S_2}{3(-2+3S_1+9S_2)+27(S_1-S_2)S_2\cos(2\theta)}},\\
\\
\mathcal{V}_2=\displaystyle{\frac{8 - 27 S_1^2 + 9 (8 - 3 S_2) S_2 + 6 S_1 (4 + 9 S_2) +
 27 (S_1 - S_2)^2\cos(4 \theta)}
{288\left(1+3S_1-9S_1\cos^2(\theta)-9S_2\sin^2(\theta)\right)^2}}.
\end{array}
\end{equation}

%

\begin{proposition}\label{propequil1}
Consider the system \eqref{ss1} defined in $\{S_1-S_2> 0\} \times [-\pi/4,\pi/4]$.
\begin{itemize}
 \item [a)] The planes $\theta=\pm \pi/4$ are invariant under the flow.
 \item [b)] The singular points are: every point in the straight line $r:=\{4+3S_1+9S_2=0, \theta=0\}$ and the points $E_2=(2/3,0,-\pi/4)$ and $E_3=(2/3,0,\pi/4)$.
 \item [c)] The surfaces
 \begin{eqnarray*}
  \Pi_1:=\{S_1 + 3 S_2 + 3(S_1-S_2)\sin(2\theta)&=-\frac 4 3\},\\
  \Pi_2:=\{S_1 + 3 S_2 - 3(S_1-S_2)\sin(2\theta)&=- \frac 4 3\},
 \end{eqnarray*}
are invariant under the flow. Moreover, $\Pi_1$ contains the singular point $E_2$ and $\Pi_2$ contains the singular point $E_3$.
 \item [d)] The flow of system \eqref{ss1} is topologically equivalent to the one represented in Figure \ref{figspaceS}.
\end{itemize}
\end{proposition}
\begin{proof}
The three first statements follow by straightforward computations. Two conclude the four one, we only need to describe the dynamical behaviour of system \eqref{ss1} over the invariant plane $\theta=\pi/4$ (the case $\theta=-\pi/4$ follows similarly), since the behaviour in the rest of the phase space is conjugated with the one depicted in Figure \ref{figspace} with conjugacy given by \eqref{coord:tranx}.

Reducing system \eqref{ss1} to the invariant plane $\theta=\pi/4$ we get
\begin{equation}\label{ss2}
\begin{array}{l}
\dot{S}_1=\displaystyle{\frac{1}{3}(1+3S_1)(3S_2-3S_1+2)},\\
\\
\dot{S}_2=\displaystyle{\frac{1}{9} \left(27 S_2^2-27 S_1S_2-9S_2+3 S_1-2\right)}.\\
\end{array}
\end{equation}
It can be checked that $S_1=-1/3$ and $-S_1+3S_2+2/3=0$ are invariant lines by the flow of system \eqref{ss2}. Note that these lines corresponds with the intersection of the manifolds $\Pi_1$ and $\Pi_2$, respectively, with the plane $\theta=\pi/4$. We conclude the resting behaviour just by considering that \eqref{coord:tranx} also maps diffeomorphically the half plane $\{x-y=0, z>0\}$ into the half plane $\{S_1-S_2>0, \theta=\pi/4\}.$
\end{proof}

\begin{figure}[ht]
\begin{center}
\includegraphics{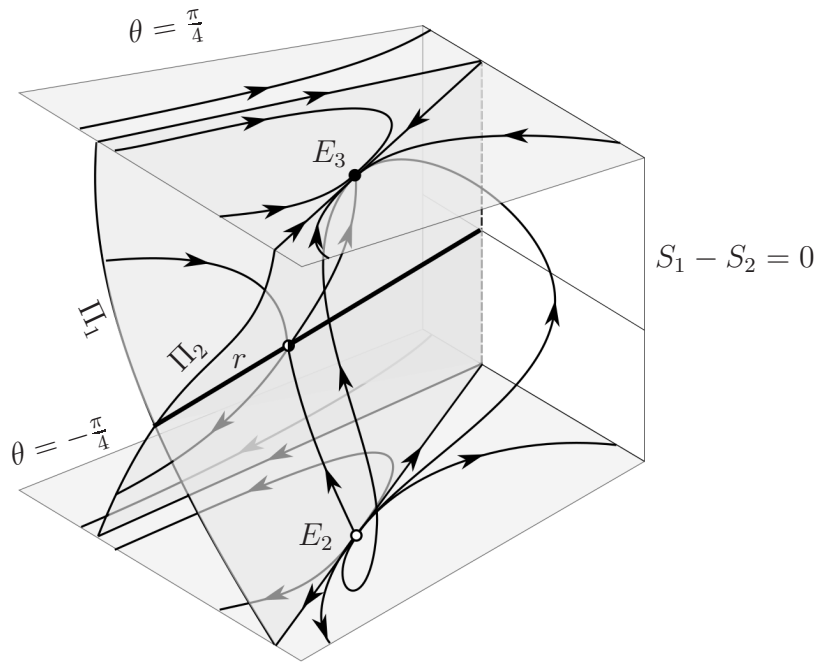}\vspace{0.35cm}
\begin{picture}(0,0)
\put(-0,150){{$S_1-S_2=0$}}
\put(-200,235){\rotatebox{10}{$\theta=\frac {\pi}4$}}
\put(-245,75){\rotatebox{17}{$\theta=-\frac {\pi}4$}}
\put(-160,112){$r$}
\put(-130,190){$E_3$}
\put(-135,45){$E_2$}
\put(-220,135){\rotatebox{-65}{$\Pi_1$}}
\put(-185,110){\rotatebox{45}{$\Pi_2$}}
\end{picture}
\caption{Global phase portrait of system \eqref{ss1} in the region $\{S_1-S_2>0\} \times [-\frac {\pi}4,\frac {\pi}4]$. Note that system \eqref{ss1} is not defined over the plane $S_1-S_2=0$. Segment $r$ in the draw is the image, by the transformation \eqref{coord:tranx}, of the continuum of non-hyperbolic singular points $r$ of system \eqref{sysabc01}, which are contanined in $\{x-y>0\}$, see Proposition \ref{propequil}. Similarly, singular points $E_2$ and $E_3$ are the respective images by \eqref{coord:tranx} of the singular points $E_2$ and $E_3$ of system  \eqref{sysabc01}. Finally, invariant surfaces $\Pi_1$ and $\Pi_2$ are the images by the same map of the invariant planes of system \eqref{sysabc01}. Therefore, $\Pi_1$ is the boundary of the $W^s(E_3)$ and $\Pi_2$ is the boundary of the $W^u(E_2)$.  }\label{figspaceS}
\end{center}
\end{figure}
\begin{remark}
System \eqref{ss1} is not defined over the plane $S_1=S_2$  due to the last equation, nevertheless the flow can be extended to this plane just by considering a vertical flow over it. Notice that this situation is compatible with the flow depicted in Figure $3$ and therefore, the change of variables from $(x,y,z)$ to $(S_1,S_2,\theta)$ can be understood as a blow up of the straight line $x=y,z=0$ into the plane $S_1=S_2$.
\end{remark}
\section{Conclusions}

We have analysed a model describing  the dynamics of  nematic liquid crystals in the presence of an imposed shear flow and in a spatially homogeneous setting. The model \eqref{eqgen1} we consider is more nonlinear in the flow effect, when compared with a previously used model, namely \eqref{eqgenfake} (studied in  \cite{sluckin,chill}). This increased nonlinearity has a positive effect in what concerns the predictions of the model,  providing that in the long time one obtains qualitatively the same behaviour as in the model without flow.

 The main driving mechanism for the long-time behaviour is the co-rotational parameter $\xi$. When this $\xi$ is  set equal to zero one obtains evolution towards rotating solutions, while for non-zero $\xi$ one has different dynamics in both short-time and long-time regimes, dynamics which can be completely described analytically and also expressed in terms of the more standard physical variables, the scalar order parameters and the director.

\section*{Acknowledgment.} A. M and A.Z. were partially supported by a Grant of the Romanian National
Authority for Scientific Research and Innovation, CNCS-UEFISCDI, project number
PN-II-RU-TE-2014-4-0657; A.Z. was also partially supported by the Basque Government through the BERC
2014-2017 program; and by the Spanish Ministry of Economy and Competitiveness
MINECO: BCAM Severo Ochoa accreditation SEV-2013-0323; 
A.E.T. was partially supported by the Spanish Ministry of Economy and Competitiveness through the project MTM2014-54275-P.

\end{document}